\documentclass[a4paper,10pt,leqno, english]{amsart}
\title[Analyitical  Signatures ]
{Analytical Signatures  and  Proper  actions }

\author{No\'{e} B\'{a}rcenas }
                \email{barcenas@matmor.unam.mx}
         \urladdr{http://www.matmor.unam.mx /~ barcenas}

 \address{Centro de Ciencias Matem\'aticas. UNAM \\Ap.Postal 61-3 Xangari. Morelia, Michoac\'an MEXICO 58089}

\author{Quitzeh Morales Mel\'{e}ndez}
 \email{qmoralesme@conacyt.mx}

 \address{CONACYT--Universidad Pedag\'{o}gica Nacional\\Camino a la Zanjita s/n, Noche Buena, Santa Cruz Xoxocotl\'an, Oaxaca, MEXICO 71230}
         
         \date{\today}

\usepackage{hyperref}
\usepackage{pdfsync}
\usepackage{calc}
\usepackage{color}
\usepackage{enumerate,amssymb}
\usepackage[arrow,curve,matrix,tips,2cell]{xy}
  \SelectTips{eu}{10} \UseTips
  \UseAllTwocells

\DeclareMathAlphabet\EuR{U}{eur}{m}{n}
\SetMathAlphabet\EuR{bold}{U}{eur}{b}{n}

\makeindex             

\theoremstyle{plain}
\newtheorem{theorem}{Theorem}[section]
\newtheorem{lemma}[theorem]{Lemma}
\newtheorem{proposition}[theorem]{Proposition}
\newtheorem{corollary}[theorem]{Corollary}

\theoremstyle{definition}
\newtheorem{definition}[theorem]{Definition}

\newtheorem{remark}[theorem]{Remark}

{\catcode`@=11\global\let\c@equation=\c@theorem}

\newcommand{\comsquare}[8]                   
{\begin{CD}
#1 @>#2>> #3\\
@V{#4}VV @V{#5}VV\\
#6 @>#7>> #8
\end{CD}
}

\newcommand{\xycomsquare}[8]                   
{\xymatrix
{#1 \ar[r]^{#2} \ar[d]^{#4} &
#3 \ar[d]^{#5}  \\
#6\ar[r]^{#7} &
#8
}
}

\newcommand{\id}{\operatorname{id}}

\newcommand{\im}{\operatorname{im}}

\newcommand{\eub}[1]{\underline{E}#1}              

\newcommand{\higherlim}[3]{{\setbox1=\hbox{\rm lim}
        \setbox2=\hbox to \wd1{\leftarrowfill} \ht2=0pt \dp2=-1pt
        \mathop{\vtop{\baselineskip=5pt\box1\box2}}
        _{#1}}^{#2}#3}

\newcommand{\version}[1]                       
{\begin{center} last edited on #1\\
last compiled on \today\\
name of texfile: \jobname
\end{center}
}

\newcounter{commentcounter}

\begin{document}

\maketitle

\begin{abstract}
In  this  paper we compare Mishchenko's  definition  of  noncommutative  signature   for  an  oriented manifold  with an  orientation preserving proper action of  a  discrete, countable  group $G$ with   the (more   analytical)  counter part  defined  by  Higson  and  Roe  in the  series  of  articles  ``Mapping  Surgery  to  analysis". A  generalization of  the  bordism  invariance  of  the  coarse index  is  also  addressed. 
 
\end{abstract}

\section{Introduction}
There are different notions of non-commutative signatures that can be applied to oriented proper cocompact $G$-manifolds for a discrete group $G$. Higson and Roe studied the relation between a signature of $C^*$-algebras, an analytic signature and the coarse index of the signature operator, they also show that these signatures are bordism and homotopy invariants.

For these definitions, they consider two types of so-called Hilbert-Poincar\'e complexes:  \textit{algebraic complexes} of finitely generated projective modules over a $C^{*}$-algebra $C$ and \textit{analytically controlled complexes} of Hilbert spaces. Both kind of complexes are required to satisfy suitable versions of Poincar\'e Duality. The \textit{algebraic} signature has values in the $K$-theory $K_{*}(C)$ of the algebra $C$, and the \textit{analytic} signature has values in the Mitchener $K$-theory of a suitable $C^{*}$-category.

All these signatures are defined for the case of a compact oriented smooth manifold $X$ and the authors showed that the analytic signature coincides with the $K$-theoretic index of the signature operator defined on the $L^{2}$-completion of the De Rham complex of $X$. In this case, it is proven that Mitchener $K$-theory coincides with the $K$-theory $K_{*}(C_{r}(G))$ of the reduced $C^{*}$-algebra of the group $G$.

Their $C^*$-algebra signature is defined for finitely generated projective
Hilbert-Poincar\'e modules over the algebra $C_{0}(X)$ of continuous functions vanishing at
infinity. In the case of an oriented smooth manifold $\tilde X$ with orientation preserving free action of a discrete group $G$ their definition makes no sense if the quotient $X=\tilde X / G$ is not compact, because the complexes considered are not finitely generated over this algebra and the representation of $C_{0}(X)$ on the given complex is not by chain maps.
The analytic signature does make sense and the proof of its coincidence with the index of the signature operator generalizes to this context.

On the other hand, Mishchenko defined a signature for  finitely generated projective algebraic Hilbert-Poincar\'e complexes over the reduced $C^*$-algebra $C_{r}^*(G)$ of the group $G$. This can be applied to a proper oriented co-compact smooth $G$-manifold $M$ with orientation preserving action of the group $G$. The analytic signature of Higson and Roe also makes sense in this context for the $L^{2}$-completion of the De Rham complex.

In this paper we show that, with slight modifications to the notion of algebraic Hilbert-Poincar\'e complex, the $C^*$-algebra signature defined by Higson and Roe coincides in even dimension with that of Mishchenko.  A consequence of this is another proof of the homotopy and bordism invariance of the signature of Mishchenko. The analytic version of the signature can be applied in this context to triangulated \textit{bounded isotropy} proper oriented $G$-manifolds of even dimension. In this case, the coincidence of the analytic signature with the coarse index of the signature operator is a consequence of the results proven by Higson and Roe. Also, another version of bordism invariance  is considered in this context. In the last section, we 
synthetize the relations between the signatures considered.



\section{Aknowledgements}

The  first author  thanks  the  support  of  PAPIIT-UNAM grants  IA 100315 "Topolog\'ia  Algebraica  y  la  Aplicaci\'on de  Ensamble  de  Baum-Connes", IA 100117 "Topolog\'ia Equivariante  y Teor\'ia  de  \'Indice", as well  as  CONACYT-SEP  Foundational  Research  grant "Geometr\'ia no  conmutativa, Aplicaci\'on de  Baum-Connes  y  Topolog\'ia  Algebraica ".

The authors thank the anonymous referee for his constructive criticism that considerably contributed to improve our inicial manuscript into the present text. The authors also thank V. Manuilov for helpful comments on the proof of Theorem \ref{maintheorem} and  Tibor  Macko for enlightening  correspondence  concerning  algebraic  versions  of  bordism.

\section{Algebraic Hilbert-Poincar\'e complexes and their signature}

In \cite{MishchenkoAlmost} a signature for a Hilbert-Poincar\'e complex was  defined.
This definition is as follows.

Let $C$ be a $C^*$-algebra. Recall that an $n$-dimensional Hilbert-Poincar\'e  complex is a triple $(E, b, S)$ where $(E,b)$ is an $n$-dimensional chain complex
\begin{equation}
\xymatrix{
E_{0}&\ar[l]^{b_{1}} E_{1}&\ar[l]^{b_{2}}\cdots &\ar[l] E_{n-1}&\ar[l]^{b_{n}} E_{n} \\
 }
\end{equation}
of finitely generated projective Hilbert modules over a $C^*$-algebra $C$, the $b_{k}, k=1, \dots, n$ are bounded adjointable maps, $b=\oplus_{k}b_{k}:E\to E$, $E = \oplus_{k}E_{k}$,  
and $S:E \to E$ is a self-adjoint operator such that 
\begin{enumerate}
\item $S_k: E_{n-k} \to E_{k}$, where $S_k=S\vert_{E_{n-k}} $,
\item $b_{k}S_{k}+S_{k-1}b^{*}_{n-k+1}=0$ and 
\item $S$ induces an isomorphism from the homology of the dual complex  
$(E,b^*)$ to the homology of the complex  $(E,b)$.
\end{enumerate}
The second condition means that $S:(E,-b^*)\to (E,b)$ is a chain map. 

We recall the following definition.
\begin{definition} (conf. \cite[def.2.2,p.280]{HigsonRoe1})\label{Mappingcone} 
The mapping cone of a chain map $A:(E',b') \to (E,b)$ is the complex
\begin{equation}
\xymatrix{
E''_{0}&\ar[l]^{b_{1}} E''_{1}&\ar[l]\cdots &\ar[l]^{b_{n}} E''_{n}&\ar[l]^{b_{n+1}} E''_{n+1} \\
 }
\end{equation}where $E''_{j}=E'_{j-1} \oplus E_{j}$ and differential $b'':E''\to E''$ defined by
\begin{equation}
b''_{j}=\left(\begin{array}{cc}
-b'_{j-1}  & 0      \\
A_{j-1} & b_{j}
\end{array}\right)
\end{equation}
\end{definition}


Using the language and notations in \cite{HigsonRoe1}, the definitions of the signature are as follows.

\begin{definition}\label{AsmishSignature}(Mishchenko, \cite[sec.3]{MishchenkoAlmost}). 
Let $(E, b, S)$ be a Hilbert-Poincar\'e complex of Hilbert $C$-modules (with $S$ self-adjoint and $bS+Sb^*=0$) and let $(E\oplus E, b_S)$ the mapping cone of $S$.
Then, the signature of $(E, b, S)$ is the formal difference $[Q_+]-[Q_-]$ in $K_0(C)$ 
of the positive and negative projection of the restriction of the map 
$B_S=b_{S}^{*}+b_{S}$ to the $+1$ eigenspace of the symmetry which exchanges the two copies of $E$ in $E\oplus E$.
\end{definition}

\begin{remark}
In the previous definition we made use of the fact that the self-adjoint operator $B_S=b_{S}^{*}+b_{S}$ is invertible. Indeed,  property (iii) in the definition of the Hilbert-Poincar\'e  complex is equivalent to the acyclicity of the complex $E\oplus E$ by lemma 2.3 of \cite{HigsonRoe1}. This is equivalent to the invertibility of $B_{S}$ according to proposition 2.1 of \cite{HigsonRoe1}.
\end{remark}

\begin{remark}\label{indexoperator} In the construction of Mishchenko \cite{MishchenkoAlmost} the summands in the mapping cone are interchanged and this gives a different formula for the operator: $B_S=b_S+Tb_ST$, where $T$ is the symmetry interchanging the two summands copies of $E$ in $E\oplus E$ (see \cite[p.14]{MishchenkoAlmost} and notice the typo in the identity $H_k = T^*_kH_{n-k+2}T_{k-1}$). In the notation used here, Mishchenko's definition of the mapping cone would be the complex
\begin{equation}
\xymatrix{
E''_{0}&\ar[l]^{b_{1}} E''_{1}&\ar[l]\cdots &\ar[l]^{b_{n}} E''_{n}&\ar[l]^{b_{n+1}} E''_{n+1} \\
 }
\end{equation}where $E''_{j}=E_{j} \oplus E'_{j-1}$ and the differential $b'':E''\to E''$ is defined by \begin{equation}
b''_{j}=\left(\begin{array}{cc}
b_{j}  & A_{j-1}     \\
0       &-b'_{j-1}
\end{array}\right).
\end{equation}
The signature turns out to be just the index in $K_0(C)$ of the operator \linebreak
$G_{ev}=b+b^*+S : E\to E $.
\end{remark}

\begin{definition}\label{HigsonRoeSignature}(Higson-Roe).
Let $(E, b, S)$ be an even dimensional Hilbert-Poincar\'e complex of Hilbert $C$-modules.
The signature of $(E, b, S)$ is the formal difference $[P_+]-[P_-]$ of the positive projections of 
$B+S$ and $B-S$ respectively, where $B= b + b^*$.
\end{definition}

\begin{proposition}\label{coincidence}
Definitions \ref{AsmishSignature} and \ref{HigsonRoeSignature} coincide.
\end{proposition}

\proof\, Using  remark \ref{indexoperator} one has that $P_+=Q_+$. So, it is enough to prove that the positive projection of the operator $B-S$ is equivalent to the negative projection
of $B+S=b+b^*+S$.

Consider the self-adjoint symmetry  $\varphi: E \to E$ equal to the identity on 
$E_{2i},\, i=0, \dots, 2l$ and minus the identity on $E_{2i+1},\, i=0, \dots, 2l-1$. This operator 
intertwines $B-S$ with $- B - S$ and, therefore, the positive projection of $B - S$ 
is equivalent to the negative projection of $b+b^*+S$.
$\hfill \square$

\begin{remark} 
Higson-Roe definition of the index while more elaborated is more suitable for the aim of comparison with the index of the signature operator on the de Rham complex. Mishchenko's definition is a more straightforward generalization of the signature of an algebraic Poincar\'e complex to the context of $C^{*}$-algebras.
\end{remark}


\section{Bordism invariance of the algebraic signature}

Let $(E,b)$ be an $(n+1)$-dimensional complex of Hilbert $C$-modules, $(E_0,b_0)\subset (E,b)$ an $n$-dimensional subcomplex, more precisely $E=E_1\oplus E_0$ for some $E_1$ and $b|_{E_0}=b_0$, and $S:E \to E$ is a self-adjoint operator such that 
\begin{enumerate}
\item $S_k: E_{n+1-k} \to E_{k}$, where $S_k=S\vert_{E_{n-k}} $,
\item $(b_{k}S_{k}+S_{k-1}b^{*}_{n-k+2})v$ for every $v\in E_0$ and 
\item $S$ induces an isomorphism from the homology of the dual complex  
$(E,b^*)$ to the homology of the quotient complex  $(E/E_0,b_1)$ where 
$b_1: E/E_0\to E/E_0$ is the induced boundary operator.
\end{enumerate} Such a complex is called an algebraic Hilbert-Poincar\'e complex with boundary.

Then one can show the analog to \cite[lemma 1.1,p.503]{Mishchenko}.
\begin{lemma}\label{boundaryduality}
The boundary complex $(E_0,\sqrt{-1}\,b_0,S_0)$, where 
 $S_0=bS+Sb^{*}|_{E_0}$, is an algebraic Hilbert-Poincar\'e complex (without boundary).
 \end{lemma}

\proof Consider the following diagram
\begin{equation}\label{exactsequenceboundary}
\xymatrix{
0 \ar[r] & E_0 \ar[r]^{i} & E\ar[r]^{j} & E/E_0\ar[r]& 0 \\
         &0\ar[r]& E/E_0\ar[u]^{Sj^*}\ar[r]_{j^*}& E\ar[u]_{jS}\ar[ul]^S\ar[r]_{i^{*}}&
E_0\ar[r]&0  \\
 }
 \end{equation}where the rows are exact.
We will construct a map $ b_0: E_0 \to E_0$
such that
$$
ib_0i^*=bS+Sd^{*}.
$$
On one hand the image $dS+Sd^{*}(E)$ is contained in the subcomplex
$ E_0$, i.e.
$$
bS+Sb^{*}(E)\subset E_0,
$$
so, it can
be composed with the inverse $i^{-1}:\im i \to E_0$,
and, then $S_0$ can be defined as
$$
S_0=i^{-1}[i^{-1}(bS+Sb^{*})]^{*}=i^{-1}( bS+Sb^{*})i^{-1*}.
$$(this is just chasing de diagram).

On the other hand, the subcomplex $E_0$ is a direct summand
of the complex $E$, that is, there is a subcomplex $E_1\subset E$
such that
$$
E=E_0\oplus E_1
$$
so the homomorphism $S$ is represented by a matrix, i.e.
$$S=\left(\begin{array}{cc}
S_2&F\\
F^*&S_1
\end{array}\right),\quad S^*_2=S_2,\;S^*_1=S_1,
$$where
$$S_1:E_1\to E_1,\quad S_2:E_0\to E_0,\quad
F:E_1\to E_0.$$

In this terms,
$$
i=\left(\begin{array}{c}
1\\
0
\end{array}\right),\quad
j=\left(\begin{array}{cc}
0&1
\end{array}\right),\quad
b=\left(\begin{array}{cc}
\tilde b_0&h\\
f& \tilde b_1
\end{array}\right),
$$
where $ib_0=bi,$ with $b_0$ being the differential of
the chain complex $ E_0$, i.e.
$$\left(\begin{array}{c}
1\\
0
\end{array}\right)b_0
=
\left(\begin{array}{cc}
\tilde b_0&h\\
f& \tilde b_1
\end{array}\right)
\left(\begin{array}{c}
1\\
0
\end{array}\right),
$$i.e.
$$\left(\begin{array}{c}
b_0\\
0
\end{array}\right)
=
\left(\begin{array}{c}
\tilde b_0\\
f
\end{array}\right),
$$i.e. $f\equiv0,\;\tilde b_0= b_0$, and $jb=b_1j$ with $b_1$ the
differential of the complex $E/ E_0$, i.e.
$$
\left(\begin{array}{cc}
0&1
\end{array}\right)
\left(\begin{array}{cc}
 b_0&h\\
0 &\tilde b_1
\end{array}\right)
=
b_1\left(\begin{array}{cc}
0&1
\end{array}\right)
$$which means $\tilde b_1=b_1$.

Also,
$$0=b^2=
\left(\begin{array}{cc}
 b_0&h\\
0& b_1
\end{array}\right)
\left(\begin{array}{cc}
 b_0&h\\
0& b_1
\end{array}\right)
=$$
$$=\left(\begin{array}{cc}
 b^2_0&b_0h+hb_1\\
0& b^2_1
\end{array}\right)
=
\left(\begin{array}{cc}
0&b_0h+hb_1\\
0& 0
\end{array}\right)
$$i.e.
\begin{equation}\label{chainhomotopy1}
b_0h+hb_1=0, \end{equation}
where $h:E_1 \to E_0 $.

Then,
\begin{align*}
&bS+Sb^*=\\
&=\left(\begin{array}{cc}
 b_0&h\\
0& b_1
\end{array}\right)
\left(\begin{array}{cc}
S_2&F\\
F^*&S_1
\end{array}\right)
+
\left(\begin{array}{cc}
S_2&F\\
F^*&S_1
\end{array}\right)
\left(\begin{array}{cc}
 b^*_0&0\\
h^*& b^*_1
\end{array}\right)=\\
&=\left(\begin{array}{cc}
 b_0S_2+hF^* &b_0F+hS_1\\
b_1F^*& b_1S_1
\end{array}\right)
+
\left(\begin{array}{cc}
S_2b^*_0+Fh^*&Fb^*_1\\
F^*b^*_0+S_1h^*&S_1b^*_1
\end{array}\right)=\\
&=\left(\begin{array}{rl}
 b_0S_2+S_2b^*_0+hF^*+Fh^* &b_0F+hS_1+Fb^*_1\\
b_1F^*+F^*b^*_0+S_1h^*      & b_1S_1+S_1b^*_1
\end{array}\right).
\end{align*}
From the equation $j(bS+Sb^*)=0$ we obtain
\begin{equation}\label{chainhomotopy2}
b_1F^*+F^*b^*_0=-S_1h^*,
\end{equation}
$$
b_1S_1+S_1b^*_1=0.
$$

In addition,
$$
b_0F+hS_1+Fb^*_1=(b_1F^*+F^*b^*_0+S_1h^*)^*=0,
$$
i.e.
\begin{equation}\label{chainhomotopy3}
b_0F+Fb^*_1=-hS_1,
\end{equation}
$$
bS+Sb^*=
$$
$$
=\left(\begin{array}{cc}
b_0S_2+S_2b^*_0+hF^*+Fh^* & 0\\
                       0   & 0
\end{array}\right)
$$i.e.
\begin{equation}\label{inducedduality}
S_0=
 b_0S_2+S_2b^*_0+hF^*+Fh^*.
\end{equation}

Lets now show that $S_0$ makes $ E_0$ an algebraic Hilbert-Poincar\'e complex with the differential  $\sqrt{-1}\, b_0$. 
In this case, we prove that
$b_0S_0-S_0b^*_0=0$.

Indeed,
$$b_0S_0-S_0b^*_0=$$
$$
=b_0(b_0S_2+S_2b^*_0+hF^*+Fh^*)-
(b_0S_2+S_2b^*_0+hF^*+Fh^*)b^*_0=
$$
$$
=b^2_0S_2+b_0S_2b^*_0+b_0hF^*+b_0Fh^*-
b_0S_2b^*_0-S_2b_0^{*2}-hF^*b^*_0-Fh^*b^*_0=
$$
$$
=b_0hF^*+b_0Fh^*-hF^*b^*_0-Fh^*b^*_0.
$$ 
From (\ref{chainhomotopy1}), $b_0h=-hb_1$ and
$h^*b^*_0=-b^*_1h^*$, substituting
$$b_0S_0-S_0b^*_0
=b_0hF^*+b_0Fh^*-hF^*b^*_0-Fh^*b^*_0=
$$
$$
=-hb_1F^*+b_0Fh^*-hF^*b^*_0+Fb^*_1h^*=
$$
$$
=-h(b_1F^*+F^*b^*_0)+(b_0F+Fb^*_1)h^*=
$$
$$
=hS_1h^*-hS_1h^*=0.
$$
$\hfill\square$

Now we are able to prove algebraic bordism invariance.

\begin{theorem}\label{algebraicbordisminvariance}
The signature of the boundary complex $(E_0,\sqrt{-1}\,b_0,S_0)$ is equal to zero.
\end{theorem}

\proof By adding the rows in the diagram (\ref{exactsequenceboundary}), we obtain the sequence
\begin{equation}\label{exact4term}
0\to  E_0\overset{I}{\to} (E\oplus E/ E_0 )\overset{J}{\to}(E/ E_0\oplus E)\overset{I^*}{\to}  E_0^*\to 0
 \end{equation}where
$$I=\left(\begin{array}{cc}
i\\
0
\end{array}\right),\qquad
J=\left(\begin{array}{cc}
j&0\\
0&j^*
\end{array}\right)
$$and the graduated module $A=E\oplus E/ E_0$ is a
chain complex with the differential
$$H=\left(\begin{array}{cc}
b&Sj^*\\
0&b^*
\end{array}\right)
$$, i.e.
$$
H^2=\left(\begin{array}{cc}
b&Sj^*\\
0&b^*
\end{array}\right)
\left(\begin{array}{cc}
b&Sj^*\\
0&b^*
\end{array}\right)
=\left(\begin{array}{cc}
b^2& bSj^* + Sj^* b^*\\
0&b^{*2}
\end{array}\right),
$$but
$$
bSj^* + Sb^*j^*= (bS + Sb^*)j^*=(j(bS + Sb^*))^*=0.
$$By construction, the sequence (\ref{exact4term})
is exact, i.e.
$$
\im I = \im i =\ker j = \ker J,
$$
$$
\im J=\im j \oplus \im j^*= E/ E_0 \oplus \ker i^*= \ker I^*.
$$

In terms of the decomposition $E= E_0 \oplus E_1$, and
rearranging the terms, the sequence (\ref{exact4term})
takes the form
$$
0\to  E_0\overset{I}{\to} E_0\oplus E_1\oplus E_1
\overset{J}{\to} E_0\oplus E_1\oplus E_1\overset{I^*}{\to}  E_0^*\to0
$$where
$$I=\left(\begin{array}{cc}
1\\
0\\
0
\end{array}\right),\quad
J=\left(\begin{array}{ccc}
0&0&0\\
0&0&1\\
0&1&0
\end{array}\right),\quad
H=\left(\begin{array}{ccc}
d_1&h&F\\
0&d_2&D_2\\
0&0&d^*_2
\end{array}\right)
$$

Note that the complex $E_1\oplus E_1$ is
a chain complex with the differential
$$\delta=\sqrt{-1}\left(\begin{array}{cc}
b_1&S_1\\
0&b^*_1
\end{array}\right)
$$and
 an algebraic Hilbert-Poincar\'e complex with the isomorphism 
 $T:E_1\oplus E_1\to E_1\oplus E_1$
defined by the matrix
$$
T=\left(\begin{array}{cc}
0&1\\
1&0
\end{array}\right)
$$i.e. $\delta T - T\delta^*=$
$$\left(\begin{array}{cc}
b_1&S_1\\
0&b^*_1
\end{array}\right)
\left(\begin{array}{cc}
0&1\\
1&0
\end{array}\right)
-
\left(\begin{array}{cc}
0&1\\
1&0
\end{array}\right)
\left(\begin{array}{cc}
b_1^*&0\\
S_1&b_1
\end{array}\right)
=
$$
$$=\left(\begin{array}{cc}
S_1&b_1\\
b^*_1&0
\end{array}\right)
-
\left(\begin{array}{cc}
S_1&b_1\\
b^*_1&0
\end{array}\right)
=0
$$

Now consider the diagram
\begin{equation}
\xymatrix{
E_1\oplus E_1 \ar[r]^{f} &  E_0 \\
E_1\oplus E_1 \ar[u]^{T} & E_0^*\ar[u]_{S_0}\ar[l]^{f^*}\\
 }
 \end{equation}
where the homomorphism $f:E_1\oplus E_1\to  E_0$
defined by the matrix
$f=(\begin{array}{cc}
h &F
\end{array})
$ is a chain map, i.e.
$$f\delta=(\begin{array}{cc}
h &F
\end{array})
\left(\begin{array}{cc}
b_1&S_1\\
0&b^*_1
\end{array}\right)=(\begin{array}{cc}
hb_1 & hS_1+Fb^*_1
\end{array})=$$
$$=
(\begin{array}{cc}
-b_0h & -b_0F
\end{array})=
-b_0f
$$and $S_2$ defines a homotopy  between these
maps, i.e.
$$fTf^*=
(\begin{array}{cc}
h &F
\end{array})
\left(\begin{array}{cc}
0&1\\
1&0
\end{array}\right)
\left(\begin{array}{cc}
h^* \\
F^*
\end{array}\right)
=Fh^*+hF^*,$$ but by (\ref{inducedduality}),
\begin{equation}\label{sign0}
S_0=fTf^*+b_0S_2+S_2b^*_0.
\end{equation}

In particular, this equation means that 
$$H(f)H(T)H(f^*)=H(fTf^*)=H(S_0)$$
 is an isomorphism. That is, the complex $(E_0,\sqrt{-1}\,b_0, fTf^*)$ is an algebraic Hilbert-Poincar\'e complex and the complexes $(E_0,\sqrt{-1}\,b_0, fTf^*)$ and $(E_0, \sqrt{-1}\, b_0, S_0)$ are homotopy equivalent according to definition 4.1 of \cite[p.285]{HigsonRoe1} with the homotopy given by the identity map $E_0\to E_0$. They have the same signature by theorem 4.3 in the same paper. 
 
Finally, it is obvious that the signature of the complex $(E_0,\sqrt{-1}\,b_0, fTf^*)$ is equal to zero. $\hfill\square$


\section{Analytically controlled Hilbert-Poincar\'e complexes over $C^{*}$-categories and their signature}
Here we recall the definition of an analytically controlled Hilbert-Poincar\'e complex, its signature and other relevant constructions from \cite{HigsonRoe1}.

Consider a triple $(H, b, S)$, where $(H,b)$ is an 
$n$-dimensional chain complex 
\begin{equation}
\xymatrix{
H_{0}&\ar[l]^{b_{1}} H_{1}&\ar[l]^{b_{2}}\cdots &\ar[l] H_{n-1}&\ar[l]^{b_{n}} H_{n} \\
 }
\end{equation}
of Hilbert spaces, the operator $b= \oplus_k b_k : H \to H$, where $H=\oplus_k H$, is an unbounded, closed operator such that 
$b\circ b$ is defined an equal to zero, i.e. 
$Image(b)\subset Domain(b),\; b^{2}=0$.
The map $S:H \to H$ is an everywhere defined self-adjoint operator such that 
\begin{enumerate}
\item $S_k: H_{n-k} \to H_{k}$, where $S_k=S\vert_{H_{n-k}}$;
\item $S:(H,-b^*)\to (H,b)$ is a chain map, i.e. $S(Domain(b^{*}))\subset Domain(b)$ and
$(bS+Sb^{*})v=0$ for every $v\in Domain(b^{*})$;
\item $S$ induces an isomorphism from the homology of the dual complex  
$(H,b^*)$ to the homology of the complex  $(H,b)$.
\end{enumerate} Such a triple is called an \textit{analytic Hilbert-Poincar\'e complex}.

\begin{definition}
A $C^{*}$-category $\mathfrak{A}$ is an additive subcategory of all Hilbert spaces and bounded linear maps which is closed under taking adjoint of morphisms and such that the morphisms sets $\mathop{Hom}_{\mathfrak{A}}(H_{1}, H_{2})$ are Banach subspaces of the set $\mathop{Hom}(H_{1}, H_{2})$ of bounded linear operators from the Hilbert space
$H_{1}$ to the Hilbert space $H_{2}$.
\end{definition}

\begin{definition}
A $C^{*}$-category ideal $\mathfrak{J}$ of the $C^{*}$-category $\mathfrak{A}$ 
is a $C^{*}$-subcategory possibly without identity morphisms
such that any composition of a morphism in $\mathfrak{A}$ with a morphism in
$\mathfrak{J}$ is a moprhism in $\mathfrak{J}$.
\end{definition}

\begin{remark}
In the case of a $C^{*}$-category with a single object, this definition of ideal 
coincides with that of a (bilateral) ideal of a $C^{*}$-algebra of bounded operators 
on a fixed Hilbert space. In all of the following constructions, we  will  fix  this Hilbert space.
\end{remark}

\begin{definition}\label{controlledoperator}
An unbounded, self-adjoint Hilbert space operator 
$D:H \to H$ is said to be analytically controlled over 
the pair ($\mathfrak{A}$, $\mathfrak{J}$) if
\begin{enumerate}
\item $H$ is an object of  $\mathfrak{J}$,
\item the operators $(D\pm iI)^{-1}$ are morphisms of $\mathfrak{J}$, and
\item the operator $D(1+D^{2})^{\frac{1}{2}}$ is a morphism of $\mathfrak{A}$.
\end{enumerate} 
\end{definition}

This definition means that $f(D)$ is a morphism of $\mathfrak{J}$ for every
$f\in C_{0}(\mathbb{R})$ and $f(D)$ is a morphism of $\mathfrak{A}$ for every
$f\in C_{0}[-\infty,\infty]$.

\begin{definition}\label{controlledcomplex}
A complex $(H, b)$ of Hilbert spaces 
is said to be analytically controlled over ($\mathfrak{A}$, $\mathfrak{J}$) if 
the self-adjoint operator $B= b+b^{*}$ is analytically controlled over ($\mathfrak{A}$, $\mathfrak{J}$) according to  definition \ref{controlledoperator}.
\end{definition}

\begin{definition}\label{controlledHPcomplex}
An analytic Hilbert-Poincar\'e complex $(H, b, S)$ is said to be analytically 
controlled over ($\mathfrak{A}$, $\mathfrak{J}$) if the complex $(H, b)$ is 
analytically controlled over ($\mathfrak{A}$, $\mathfrak{J}$) in the sense of 
the previous definition, i.e. if $B= b+b^{*}$ is analytically controlled, and
the duality operator $T$ is a morphism in $\mathfrak{A}$.
\end{definition}

It is shown in \cite[lemma 5.8 and the discussion on p.291]{HigsonRoe1} that 
for a Hilbert-Poincar\'e complex 
analytically controlled over ($\mathfrak{A}$, $\mathfrak{J}$) the difference 
$P_{+}-P_{-}$ of the positive projections of the operators $B+S= b+b^{*}+S$
and $B-S= b+b^{*}-S$ (respectively) belongs to the ideal $J$ of $A$, where
$A$ is the $C^{*}$-algebra of  $\mathfrak{A}$-endomorphisms of the space $H$ and
$J$ is the  $C^{*}$-algebra of  $\mathfrak{J}$-endomorphisms of the same space.
This means that the formal difference $[P_{+}]-[P_{-}]$ is an element of the 
group $K_{n}(J)$. There is a natural map $K_{n}(J)\to K_{n}(\mathfrak{J})$ so 
there is a class in $K_{n}(\mathfrak{J})$ determined by the difference 
$[P_{+}]-[P_{-}]$.

\begin{definition}\label{ Definition_algsignature}
Let  $(H, b, S)$  be a Hilbert-Poincar\'e complex analytically 
controlled over ($\mathfrak{A}$, $\mathfrak{J}$). Its analytical signature is the class 
determined by the formal difference $[P_{+}]-[P_{-}]$ in $K_{n}(\mathfrak{J})$.
\end{definition}

\section{Signatures of a  $G$-manifold}

In this section we modify some the notions in \cite{HigsonRoe2} to extend the main 
results there to include proper, not necessarily free actions. Namely, 
we extend the definitions of the control categories. Also, we take into account 
the additional structure on the complex $C^{l^2}_*(M)$ of an oriented co-compact 
$G$-manifold $M$ (with orientation-preserving $G$-action) needed to make it an algebraic Hilbert-Poincar\'e complex over the reduced $C^{*}$-algebra $C^{*}_r(G)$. 
This complex also is interpreted by Higson and Roe as an
analytically controlled Hilbert-Poincar\'e complex and, therefore, it has two signatures.
The relation between this signatures is addressed. 

\subsection{The algebraic signature of a triangulated smooth $G$-manifold}

In \cite{Illman}, \cite{Korppi} it is shown that a smooth manifold $M$ with proper action of a discrete group $G$ admits $G$-invariant triangulations. It is also shown the uniqueness of this piece-wise linear structure up to barycentric subdivision. In this case one shall choose a triangulation such that every simplex is either fixed point-wise or permuted by the action. 

Following \cite[p.306-310]{HigsonRoe2} we denote by $C_{*}(M)$ the space of 
finitely supported simplicial chains on $M$ with complex coefficients. Then, 
for each $p$ the complex vector space $C_{p}(M)$ has a basis comprised of  the 
$p$-simplices on $M$. Define an inner product on $C_{p}(M)$ such that this basis is orthonormal. The 
completion of this space is denoted by $C^{l^{2}}_{p}(M)$, in other words, this is the Hilbert space of square summable $p$-chains on $M$. The differentials 
$\partial_{p}: C_{p}(M)\to C_{p-1}(M)$ extend to operators 
$b_{p}: C^{l^{2}}_{p}(M)\to C^{l^{2}}_{p-1}(M)$.

The operators $b_{p}$ are bounded if the number of simplices in the triangulated 
space $M$ with a common boundary is bounded, and this assumption can in turn be reduced to requiring that the number of simplices containing a point in the space $M$ is bounded. 
Such space $M$ is called of \textit{bounded geometry}. 

Also, the adjoint operators $b^{*}_{p}: C^{l^{2}}_{p-1}(M)\to C^{l^{2}}_{p}(M)$ 
identify with the extension of the co-boundary maps. This makes $(C^{l^{2}}_{*}(M), b)$ 
a complex of Hilbert spaces.

Denote by $C_{0}(M)$ the algebra of continuous functions vanishing at infinity.
Define a representation of $C_{0}(M)$ on $C^{l^{2}}_{*}(M)$ as follows: for
every $f\in C_{0}(M)$ and chain $c= \sum_{\sigma}  c_\sigma[\sigma]$, 
$$
f\cdot c= \sum_{\sigma} f(b_\sigma)c_\sigma[\sigma],
$$where $b_\sigma$ is the barycenter of the simplex $\sigma$.
With this and the bounded geometry assumption, one might interpret 
$(C^{l^{2}}_{*}(M), b)$ as a complex of Hilbert $C_{0}(M)$-modules, but these spaces 
are not in general finitely generated over this algebra, and the representation of 
$C_{0}(M)$ on $C^{l^{2}}_{*}(M)$ is not by chain maps.

On the other hand, as the action of $M \times G\to M$ is simplicial, the complex 
$C_{*}(M)$ has a natural action of this group defined by the formula
$$
c\cdot g= \sum_{\sigma} c_\sigma[\sigma]g =\sum_{\sigma} c_\sigma[\sigma g],
$$ for $g\in G$. The action is simplicial, so it commutes with the boundary map.
As the action either fixes simplices or permutes them, this action
is by unitaries, and it extends to a representation of the 
reduced $C^{*}$-algebra $C^{*}_r(G)$ of the group $G$. This means that
$(C^{l^{2}}_{*}(M), b)$ is a complex of $C^{*}_r(G)$-modules.
If the quotient $M/G$ is compact, then the modules  
$C^{l^{2}}_{p}(M)$ are finitely generated: one may assume that there is
a finite number of simplices in the triangulation of the compact quotient $X=M/G$ induced 
by the map $M\to M/G$, and this means that there is a finite number of 
$G$-orbits of simplices in $M$. 


In order to analyze Poincar\'e duality in this context one shall first give some 
explicit expression of the action of $G$ on cochains. If $u: C_{p}(M)\to \mathbb{C} $ a 
 $p$-cochain, this is defined by the rule
$$
(u\cdot g)[\sigma]=u([\sigma g^{-1}]), 
$$ for a simplex $\sigma \in C^{p}(M)$.

The Poincar\'e duality homomorphism of an oriented, possibly non-compact 
manifold $M$ is given by the intersection 
$[M]\cap u$ of the fundamental class of the manifold with a finitely supported 
cochain $u$. More precisely, let $u: C^{n-p}(M)\to \mathbb{C} $ be a finitely supported
$(n-p)$-cochain and $[M]=\sum_{\sigma}  (-1)^{\epsilon(\sigma)}[\sigma]$ be the fundamental class, where $\epsilon(\sigma)$ denotes the orientation of the simplex 
$\sigma$ induced by the orientation of the manifold $M$, and the sum runs over all $n$-simplices in the triangulation of $M$. Then, the 
Poincar\'e duality homomorphism $T_{p}:C^{n-p}(M)\to C^{p}(M)$ is defined by the formula
$$T_{p}(u)=[M]\cap u=\sum_{\sigma=[v_{0}\cdots v_{n}]}  (-1)^{\epsilon(\sigma)}u([v_{0}\cdots v_{n-p}])
[v_{n-p}\cdots v_{n}].$$
This map is $G$-equivariant, i.e. satisfies the identity 
$$
T_{p}(u\cdot g)=(T_{p}(u))\cdot g.
$$
Indeed, 
\begin{align*}
(T_{p}(u))\cdot g &= \left(\sum_{\sigma=[v_{0}\cdots v_{n}]}  (-1)^{\epsilon(\sigma)}u([v_{0}\cdots v_{n-p}])
[v_{n-p}\cdots v_{n}]\right)\cdot g=\\
&=\sum_{\sigma=[v_{0}\cdots v_{n}]}  (-1)^{\epsilon(\sigma)}u([v_{0}\cdots v_{n-p}])
[v_{n-p}\cdots v_{n}]g=\\
&=\sum_{\sigma=[v_{0}\cdots v_{n}]}  (-1)^{\epsilon(\sigma)}u(([v_{0}\cdots v_{n-p}]g)g^{-1})
[v_{n-p}\cdots v_{n}]g=\\
&=\sum_{\sigma=[v_{0}\cdots v_{n}]}  (-1)^{\epsilon(\sigma)}((u\cdot g) [v_{0}\cdots v_{n-p}]g)
[v_{n-p}\cdots v_{n}]g=\\
&=\sum_{\sigma g=[v_{0}\cdots v_{n}]g}  (-1)^{\epsilon((\sigma g )g^{-1})}(u\cdot g) [v_{0}\cdots v_{n-p}]
[v_{n-p}\cdots v_{n}]=
\\\\
&=\sum_{\gamma=[w_{0}\cdots w_{n}]}  (-1)^{\epsilon(\gamma )}(u\cdot g) [w_{0}\cdots w_{n-p}]
[w_{n-p}\cdots w_{n}]=\\\\
&=T_{p}(u\cdot g).
\end{align*}where $\gamma = \sigma g$ and in the last step we have used the identity 
$ \epsilon(\gamma g^{-1})=\epsilon(\gamma)$, that is, one must require that $g$ preserves orientation.
The equivariant map $T:C^{*}(M)\to C_{*}(M)$ satisfies the classic 
Poincar\'e duality identities and extends to a $G$-linear map 
$T:C^{l^{2}}_{*}(M)\to C^{l^{2}}_{*}(M)$. Then, if the dimension of $M$ is even, the operator 
$S: C^{l^{2}}_{*}(M)\to C^{l^{2}}_{*}(M)$ 
defined by the rule
$$S_{p}=i^{p(p-1)}T_{p}: C^{l^{2}}_{n-p}(M)\to C^{l^{2}}_{p}(M)$$
satisfies the properties: 
\begin{enumerate}
\item $S$ is self-adjoint,
\item $bS+Sb^{*}=0$ and 
\item $S$ induces an isomorphism from the homology of the dual complex  
$(C^{*}(M),b^*)$ to the homology of the complex  $(C^{*}(M),b)$.
\end{enumerate} Therefore, $(C^{*}(M),b,S)$ is an algebraic 
Hilbert-Poincar\'e complex over $C^{*}_r(G)$ and has algebraic signature 
in $K_{0}( C^{*}_r(G))$ as in definitions  \ref{AsmishSignature} or 
\ref{HigsonRoeSignature}.

With this structure, one obtains another proof of the following:

\begin{proposition}
The signature of Mishchenko is a homotopy invariant.
\end{proposition}

\proof This is theorem 4.3 of \cite{HigsonRoe1} applied to the signature defined there and recalled here as \ref{HigsonRoeSignature}, but using
the algebraic Hilbert-Poincar\'e complex over the algebra $C_r^*(G)$ that we have just constructed. These signatures 
coincide by proposition \ref{coincidence}. 
$\hfill \square$

\begin{remark}
The construction of this Hilbert-Poincar\'e complex has been presented by Mishchenko in several conference talks before 2010, so the authors claim no originality. We refer to  
\cite[sec.3]{MishchenkoAlmost} and check that this complex satisfies the definition given there.
\end{remark}

\subsection{The analytic signatures of a smooth $G$-manifold} 
Here we generalize the $C^{*}$-categories considered in \cite{HigsonRoe2} and
reinterpret the complex $(C^{*}(M),b,S)$ as an equivariant  
analytically controlled Hilbert-Poincar\'e complex. 
Then we show that the results about invariance of the analytic signature can be applied 
to bounded geometry spaces with \textit{bounded isotropy} action, and that this is the 
case for proper spaces with bounded geometry quotient.

\begin{definition}
Let $M$ be a proper metric space. An $M$-module $H$ is a separable Hilbert space equipped with a non-degenerate representation of the $C^*$-algebra $C_0(M)$ of continuous, complex-valued functions on $M$ vanishing at infinity.
\end{definition}

\begin{definition}\label{presentation}
Let $G$ a finitely generated discrete group. A $G$-presented space $X$ is a proper geodesic metric space presented as the quotient $X=M/G$ of a proper geodesic metric space $M$ by an isometric proper action $\mu: G\times M \to M$ of the group $G$. 
The pair $(M, \mu)$ is called a $G$-presentation of $X$. 
\end{definition}

For fixed discrete  group  $G$ and space $X$,  the presentations of $X$ together with equivariant maps form a category. We avoid the action in the notation and say that $M$ is a $G$-presentation of $X$. We shall assume in the following that all such presentations have an invariant non-empty open set where the action of the group $G$ is free. 


\begin{definition}
An equivariant $G$-$X$-module is an $M$-module $H$, where $M$ is a $G$-presentation of $X$ equipped with a compatible (faithful) unitary representation of $G$.
\end{definition}

In the case of an equivariant $G$-$X$-module $M$ we will require that the representation of the $C^*$-algebra $C_0(M)$ restricts to a non-degenerate representations of the subalgebra $C_0(U)$ for a $G$-invariant non-empty open set $U\subset M$.

Given a  locally  compact, separable  and  metrizable  space,  together  with  a   non-degenerate representation  on  the  Hilbert  space $H$,  that  is,  a  nondegenerate continuous  $*$-homomorphism
$$\rho: C_0(M)\to B(H),$$
we  define  the  support of $\nu\in H$   to  be  the  complement  in $X$  of  the  union  of  all open  subsets $U\subset X$  such  that  $\rho(f)(\nu)=0$ for  all $f\in C_0(U)$. 
An  operator $T\in B(H)$ is  locally  compact on $X$  if  $fT$ and $Tf$  are  compact  operators  for  all  functions  $f\in C_0(M)$.

\begin{definition}
 
The  support  of  an  operator  $T \in B(H)$, denoted  by  $Supp(T)$,   is  the  complement  in $X\times X$ of  the  union  of  all open  subsets $U\times V \subset  X\times X$ such  that  $\rho(f)T\rho(g)=0$ for  all $f\in C_0(U)$  and  $g\in C_0(V)$. More  generally , if $C_0(X)$  and $C_0(Y)$  are non-degnerately represented   on  Hilbert  spaces  $H_X$  and $H_Y$,  then  the  support  of  a   bounded  operator $T:H_X\to H_Y$  is  the  complement  in  $Y\times X$  of  the  union  of   all  open  subsets $U\times V\subset  Y\times  X$ such  that $\rho_Y(f)T\rho_X(g)=0$ for all $f\in C_0(U)$ and $g\in C_0(V)$. 

\end{definition}

\begin{definition}
Let  $X$  be  a  locally  compact   separable  and  metrizable  space,  proper  in the  sense  of  metric  geometry, meaning  that  closed  balls  are  compact. Let $\rho: C_0(X)\to B(H)$  be  a nondegenerate  representation on the  Hilbert space $X$. 

An  operator $T\in B(H)$  is  boundedly   controlled  if  the  support $Supp(T)$ is   at  bounded  distance  of  the  diagonal   in $X\times X$,  that  means 
$$\sup_{y\in Supp(T), x \in \Delta(X)} \{ (d_{X\times X}(y, x) \}<\infty. $$

An  operator $T$ is  locally  compact  on $X$ if  $fT$ and $Tf$  are  compact  for  all  functions $f\in C_0(X)$. 

Given  an  operator $T\in B(H)$,  we  define  its   propagation $Prop(T)$, to  be  the   following extended  real  number: 
$$ Prop(T)= sup \{d_{X\times X }(d(x,y)\mid  x, y\in Supp(T)\},$$ 
and   will  say  that  an operator  is  of  finite  propagation  if this  number  is   finite. 

\end{definition}

\begin{definition}
The  category  
$\mathfrak{A}(G, M)=\mathfrak{A}(X, G, M)$   is  the  category  where  the  objects are  
equivariant $G$-$X$-modules for a fixed presentation $M$,  and the  morphisms  are norm  limits  of  $G$-equivariant,  bounded,  finite  propagation, pseudolocal operators between $G$-$X$-modules. The  ideal  $\mathfrak{C}(G, M)=\mathfrak{C}(X,G, M)$  is  the category  with  the  same  objects  as $\mathfrak{A}(G, M)$, and   morphisms   given  by  norm  limits  of $G$-equivariant, bounded, locally compact operators. 
\end{definition}

The category $\mathfrak{A}(G, M)$  and its ideal
$\mathfrak{C}(G, M)$,  are  defined in  an  analogous  way  to  the categories 
$\mathfrak{A}(X) $ and $\mathfrak{C}(X)$, compare \cite[p. 304]{HigsonRoe2}.

\begin{definition}
Let $X$ be  a  proper (both  in  the  sense  of  group actions and  metric  geometry), locally  compact and  metrizable $G$-space.  $D:H\to H $ be  a  bounded  selfadjoint  operator .   We  will  say  that $T$  is  analytically  controlled  if it is controlled over 
$(\mathfrak{A}(X,G,M), \mathfrak{C}(X,G,M))$ in the sense of definition 
\ref{controlledoperator}.

\end{definition}

We  will  now  include  for  the  sake  of  completness  the  following  notion  of   geometrically  controlled operator, which  will  be  relevant  for  the comparison  with  Hilbert-Poincar\'e  complexes (see  Def. 5.3 and 5.5 in \cite{HigsonRoe1}  for more  details  on  geometric  
control): 

\begin{definition}
Let  $X  $  be  a  geodesic, proper space (in the  sense  of  metric  geometry, meaning  that  closed balls  are  compact). A  complex  based  vector  space $V$ is  called  geometrically  controlled  over $X$  if  it  is  provided  with a  basis $B\subset V$,  and a  function $c:V\to X$ with  the  following  property: for  every $r>0$, there is  an $N<\infty $  such  that if $S\subset X$  has diameter  less  than  $R$, then $c^{-1}(S)$  has  cardinality  less  than $N$. 

A linear  transformation $T:V\to W$  between  geometrically  controlled  spaces is  geometrically   controlled  if  
\begin{itemize}
\item The   matrix  coefficients  with  respect  to  the  basis  are  uniformly  bounded.
\item There  exists some $C>0$ such  that  the  $(v,w)$-matrix  coefficient  is  zero  whenever $d(c(v), c(w))>C$. 
\end{itemize}

\end{definition}


For $X$ compact, one can now proof the analogous of lemma 2.12 in \cite{HigsonRoe2}:

 \begin{lemma}\label{CategoryEquivalence}
Let $M$ be a proper (metric) space and $G\times M\to M$ be a proper effective action 
with $X=M/G$ a compact space. Assume that $M$ is provided with a $G$-invariant measure which is  finite on compact subsets. Then the $C^*$-algebra of endomorphisms of a non-trivial 
object in $\mathfrak{C}(X, G, M)$ is Morita equivalent to $C^*_r(G)$ and, therefore, their $K$-theories are isomorphic.
  \end{lemma}
  
\proof \, We will follow two steps in the proof: 

\begin{enumerate}
\item Every non-trivial $G$-$X$-module $H$ with an effective representation of $G$ by unitaries contains a non-trivial subspace which can be endorsed with the structure of a Hilbert module over $C_r(G)$ whose algebra of compact operators in the sense of Hilbert modules is isomorphic to the algebra of endomorphisms of $H$ in the category $\mathfrak{C}(X, G, M)$.
\item There  is an example of a Hilbert $G$-$X$-module $H$ such that this algebra of compact operators is isomorphic to $C_r(G)$.
\end{enumerate}

Step (i). Let $H$ be a non-trivial $G$-$X$-module for the presentation $M$ (i.e. a $C_0(M)$-module) with a non-trivial representation of $C_0(M)$ and compatible effective representation of the group $G$ by unitary operators in the sense that $(f\cdot v)\cdot g = f^g\cdot (v \cdot g)$ for every $f\in C_0(M), v\in H, g\in G$, where $f^g(x)=f(gx), x\in M$. Recall that a $v\in H$ is said to be compactly supported if there is a function $f\in C_c(M)$ acting as the identity on $v$, i.e. $f\cdot v=v$. Denote by $H_{C_c(M)}$ the vector space of such elements. This space is non-trivial: take $f\in C_c(M)$ and $v\in H$ such that $f\cdot v\neq 0$ and choose $h\in C_c(M)$ such that $hf=f$, then one has $h\cdot (f\cdot v)=(hf)\cdot v=f\cdot v$. 

Define a $\mathbb{C}[G]$-valued inner product on $H_{C_c(O)}$ by 
$$
\langle\langle v, w \rangle\rangle= \sum_{g\in G} \langle v\cdot g, w \rangle[g],
$$ where $O\subset M$ is a non-empty $G$-invariant open subset such that the action 
$G\times O\to O$ is free. Note that, if $supp(f)\cap supp(h)=\emptyset $ then $fh=0$. Assume that $supp(f)\subset U$ for some $U \subset O$ such that $gU\cap U=\emptyset$ for $g\neq 1$. This means that  $supp(f)\cap supp(f^g)=\emptyset $, therefore 
\begin{align*}
\langle\langle v, v \rangle\rangle &=\sum_{g\in G} \langle v\cdot g, v \rangle[g]=\\
                                                   &=\sum_{g\in G} \langle (f\cdot v)\cdot g, f\cdot v \rangle[g]=\\ 	  						&=\sum_{g\in G} \langle f^g\cdot (v\cdot g), f\cdot v \rangle[g]=\\
&=\sum_{g\in G} \langle (\bar f f^g)\cdot (v\cdot g), v \rangle[g]=\\
&=\langle (\bar f f)\cdot v, v \rangle[1]=\langle f\cdot v, f\cdot v \rangle[1]=
\langle v,  v \rangle[1]=[1],
\end{align*} if $\|v\|=1$. Here we have used that the adjoint of the operator $f: H \to H$ is the conjugate $\bar f$ of this function, which has the same support.

Similarly
\begin{align*}
\langle\langle v, v\cdot g \rangle\rangle &=\langle (\bar f^g f^g)\cdot (v\cdot g), v\cdot g \rangle[g]=\langle f^g\cdot (v\cdot g), f^g\cdot (v\cdot g)\rangle[g]=\\
&=\langle v\cdot g,  v \cdot g\rangle[g]=\langle v,  v \rangle[g]=[g].
\end{align*} 
This computations show that every element in the group algebra $\mathbb{C}[G]$ can arise as the inner product of some elements in $H_{C_c(O)}$.

Then one checks positivity, completes simultaneously $\mathbb{C}[G]$ to $C_r(G)$   and $H_{C_c(O)}$ to a Hilbert $C_r(G)$-module  $H_{G}$ and writes 
$$
\overline{H_{C_c(O)}}=H_{G}\otimes_\lambda l^2(G)
$$where $\lambda: C_r(G)\to \frak{B}(l^2(G))$ is the left regular representation.
Note that, as $C_c(O)$ is dense in $C_0(O)$ and $C_0(O)H$ is dense on $H$, one has that  
$C_c(O)H$ is also dense in $H$ and $\overline{H_{C_c(O)}}=H$.

Then one proves by analogy with lemma 2.2 and lemma 2.3 of \cite[p.243,244]{roeComparing} that the algebra of compact operators of $H_{G}$ in the sense of Hilbert modules is isomorphic to the algebra of endomorphisms of $H$ in the category $\mathfrak{C}(X, G, M)$.
 
 Step (ii). The example is the Hilbert space is the completion $L^2(M)$ of $C_c(M)$ with respect to the norm defined by the complex-valued inner product 
 $$
 \langle f, g \rangle = \int \bar f(x) h(x)d\mu(x)
 $$ where $\mu$ is a $G$-invariant measure finite on compact subsets. The proof that this is an example are precisely  lemma 2.2 and lemma 2.3 of \cite[p. 243,244 ]{roeComparing}.
 $\hfill\square$

Let $M$ be a simplicial complex, and  let  $G\times M \longrightarrow M$ be  a proper simplicial action of a discrete group $G$.  Assume that the quotient $M/G$ is compact. Let $\mathfrak{F}_M$ the family of (finite) subgroups of $G$ having non empty fixed point set in $M$, i.e. 
$$\mathfrak{F}_M= \{H < G\, \vert \, M^{H}\neq \emptyset \},$$
where
$$M^{H}=\{x\in M\, \vert \, hx=x, \text{ for every } h\in H \}.$$

\begin{definition}
The action $G\times M \longrightarrow M$ is said to be of bounded isotropy if the order of the elements in
$\mathfrak{F}_M$ is uniformly bounded, i.e. there is a constant $c_M$ such that $\vert H \vert < c_M$ for 
every $H\in \mathfrak{F}_M$. 
\end{definition}

\begin{lemma}
If the quotient $X=M/G$ is of bounded geometry and the action $G\times M \longrightarrow M$ is of bounded isotropy, then $M$ is of bounded geometry.
\end{lemma}
\proof Take a point $x \in M$ and let $S(x)$ the set of simplices containing $x$. Denote by $p: M \longrightarrow M/G$ the projection on the quotient. Then $p(S(x))=S(p(x))$ and, therefore, 
$\#S(x)\leqslant \#S(p(x))\cdot c_M\leqslant N\cdot c_M$, where $N$ is the bound on the number of simplices containing a point in $M/G$. $\hfill\square$

\begin{lemma}
A proper space $M$ with proper, co-compact proper action $G\times M \to M$ 
of a discrete group $G$ is of bounded isotropy.
\end{lemma}
\proof Choose finite family  $(U_{i},G_{i}), \; i=1, \dots, N'$ such that $U_{i}\subset M$
are open subsets and $G_{i}< G$ are finite subgroups such that, if a point $x \in gU_{i}$ for some $g\in G$, then one has that $ G_{x}< gG_{i}g^{-1}$. Therefore 
$\beta= \max_{i} |G_{i}|$ is a bound on the orders of isotropy groups of points in 
the space $M$. $\hfill\square$

We  will  not  discuss the   functorial  properties  of  the $C^*$-alebras  associated  to coarse  structures  of  a proper  metric space, called  in the   literature  ``morphism covering a  coarse  map".  However,  we  will  need  a  restriction map for  the  inclusion  of  a boundary component   into  a  a  bordism   satisfying  some  additional  assumptions,  see the  coments  preceeding  section \ref{sectionbordism}.

\begin{definition}
Let  $X$  be   a  proper  space and $M$ a $G$-presentation of $X$.   A Hilbert-Poincar\'e complex is equivariantly analytically controlled if it is analytically controlled over 
$(\mathfrak{A}(X, G, M),\mathfrak{C}(X, G, M))$, i.e. the modules in the complex are objects of these categories, the operator $B=b+b^*$ is controlled over $(\mathfrak{A}(X, G, M),\mathfrak{C}(X, G, M))$ and the duality operator $S$ is a morphism in the category $\mathfrak{A}(X, G, M)$.
\end{definition}

In the following, by controlled in the case of a complex of Hilbert modules we mean equivariantly analytically controlled and in the case of an operator we mean controlled over $(\mathfrak{A}(X, G, M),\mathfrak{C}(X, G, M))$.


\begin{theorem}\label{HomBordInvariance}
If the quotient $X=M/G$ is of bounded geometry and the action $G\times M \longrightarrow M$ is of bounded isotropy and orientation preserving, then its Higson-Roe non-commutative signature is a homotopy and bordism invariant in the controlled category.
\end{theorem}

\proof As $M$ is of bounded geometry, its simplicial chain and cochain complexes are geometrically controlled. The action either permutes or fixes simplices and is therefore unitary, and the fundamental cycle of such a triangulation is invariant. By theorem 3.14 in \cite[p.309]{HigsonRoe2} geometric control implies analytic control. The comment before section 3.2 on \cite[p.310]{HigsonRoe2} ensure that this is true also in the equivariant setting. 
This means that the $l^2$-chain complex $C^{l^2}_*(M)$ of $M$ is an example of an 
analytically controlled Hilbert-Poincare complex. 

In the case of bordism invariance, one shall assume that one has a triangulated bordism such that the simplices in the boundary coincide with the given triangulation of $M$. 

The result now follows as a corollary of theorems 5.12 and 7.9 of \cite{HigsonRoe1}. $\hfill\square$


\section{Bordism invariance of  the  coarse  index}\label{sectionbordism}
In  this  section, we  review  the  approach  to bordism  invariance  of  the  coarse  index  due to C. Wulff \cite{wulffdiplomarbeit} and  extend it  to  the  context  of  manifolds  with  proper actions  of  a  discrete  group.  This  section's  results  benefited in  a  fundamental  way  from  remarks  of  an anonymous  referee. The  authors  thank  her  or  him. 

We  recall  that   given  a  smooth manifold  with  a  proper, smooth  $G$-action $M$,  the  existence  of $G$- invariant  Riemannian  metrics due  to  Palais \cite{palais} implies  the  existence  of  a $G$-invariant  geodesic  length  metric on $M$. Recall  that  this  geodesic  metric  is proper  in  the  sense  of  metric  geometry, meaning  that  closed balls  are  compact.  If  the   original  manifold  is  geodesically  complete, then  so  is  the  one  with  the  $G$-invariant metric.  In the  case  of  a   non co-compact  manifold  $M$,  there  might   be  many  quasi-isometry  classes  of $G$-invariant metrics  on $M$.  We  will fix, however,  a boundedly  controlled coarse structure  coming  from a  particular $G$-invariant, complete   geodesic  metric structure   for  the  remain  of  this  section.  

In  order  to  define  adequately   the  (coarse index)  boundary maps  and  the  functoriality  properties  after  $K$-theory,  certain  remarks  on  the bounded coarse structure on a proper  geodesic manifold are  pertinent.   References for  the  bounded coarse  bounded structure, and  other  ones  defined  on a geodesic  metric  space include \cite{HigsonRoeAnalytic}, chapter 6,  although we  specialize  here  to  the Riemannian  manifold  case.

\begin{definition}[Coarse map  in  the  bounded metric  structure]
Let $M$ and $N$ be  proper Riemannian $G$-manifolds equipped  with   $G$-invariant  geodesic length metrics $d_M$ and $d_N$.      
A map  $f:N\to M$ is  a  coarse  map  if
\begin{itemize}
\item The  inverse image  of  every closed ball  is  compact. 
\item For  every $R>0$, there  exists $\delta>0$, such that $d_N(x, x^{'}) < R$ implies $d_M(f(x), f(x^{'}))<\delta $. 
\end{itemize}
\end{definition}

A coarse  map  induces a $C^{*}$-homomorphism  of  the alebras  of  locally  compact  and  finite propagation operators  by  lemma 6.3.12 in \cite{HigsonRoeAnalytic}.

\begin{definition}[Referenced Manifolds ]
 Let $M$,   $N_1$ and  $N_2$ be  proper, oriented  $G$-manifolds of  dimension  $n$ furnished  with  an orientation preserving $G$-action. 
 
  Assume  that  $M$, $N_1$ and $N_2$ are  furnished  with  the bounded   coarse  structure associated  to a $G$-invariant  geodesic length metric. The manifolds $N_1$ and $N_2$ are referenced  manifolds  with  respect  to $M$  if  they  are  furnished  with  a $G$-equivariant  coarse  map $i_1:N_1 \to  M$  and $i_2:N_1\to M$
 \end{definition}
 
 \begin{definition}[Referenced Bordism]
 
  Let $N_1 $ and  $N_2$ be   referenced  manifolds  with  respect  to  $M$. 
 
 A  referenced  bordism  from  $(N_1, f_1) $ to  $(N_2, f_2)$  is  a  referenced $G$-manifold $W$, together  with a  coarse  map $F: W\to M$,  such that  there exists  a  positive  real number $K$  with  the  property that  the  diagram  depicting  the  inclusions  of  the  boundary  components $j_i: N_i\to \partial W$,  
 
$$ \xymatrix{N_1 \ar[r]^-{j_1} \ar[ddr]_-{f_1}   & \partial W \ar[d]^{j}  &\ar[l]_-{j_2}  \ar[ldd]^-{f_2} N_2 \\ & W \ar[d]^F & \\ & M&   } $$
commutes  up  to  $K$,  meaning  that   the  inequalities
$$ d_M( f_i(n), F\circ j\circ j_i(n))<  K $$
hold  for  $i=1, 2$  and every $n\in N_i$. 
 
\end{definition}

\begin{definition}\label{c-bordism}[Analytical referenced  bordism]
Let  $M$  be  a   complete, proper   metric  $G$-space  with  an  action  of  bounded  isotropy. 
The  analytical  referenced  bordism  group  $\Omega_n ^{\rm an, eq}(M)$  is  the  group with generators 
$(N, f, E, b) $, such  that 
\begin{itemize}
\item  $N$  is  an $n$-dimensional  referenced manifold  with  respect  to $M$,  with   bounded isotropy, 

\item $f: N\to M$ is   an  equivariant coarse  map;
\item  $E$ is  a  $G$-$X$-Hilbert module  with presentation $N$, i.e. an equivariant $N$-module with $X=N/G$,  

\item   $b: E\to E$  is  boundedly   controlled  operator.   
\end{itemize}

Two  of  such  generators $(N_1, f_1,  E_1 , b_1)$ and  $(N_2,f_2,  E_2,  b_2)$
are   said  to  be   referenced-bordant with  respect  to $M$  if  there  exists  a     referenced  bordism $(W, F, E, B)$  with  respect  to $M$,    between   $N_1$   and  $N_2$,  together  with  a coarse map $F: W \to  M$,   inclusions  $j_i:N_i\to  W$,   which  induce  isometries of Hilbert  spaces  $E_i\to  E$,   and  a controlled  operator $B$,  restricting  to  $b_i$. 

\end{definition}

If the space $M$ is a proper  oriented  manifold  of  bounded  isotropy, then
one defines the  fundamental  class  in the group $\Omega_n ^{\rm an, eq}(M)$
by taking $f=\id$ and, for example, $E= \Omega^*_{L^2}(M)$, the $L^2$-completion of the de Rham complex of $M$ and $b$ as the signature operator. Although  this  is  an  unbounded  operator,  the   generalized  conditions   of  analytical  control meet (meaning  that the  Cayley  transform  is  locally  compact and  of  finite  propagation  and  the  resolvent has  finite  propagation).

One can also, take 
$E'=C^{l^2}_*(M)\oplus C^{l^2}_*(M)$ and $b=B_S$ as in (\ref{AsmishSignature}), where 
$S$ is the Poincar\'e duality homomorphism completion. Both choices coincide in terms of 
index by theorems 5.5 and 5.12 in \cite{HigsonRoe2}, using the version of analytic control 
defined in here.

\begin{definition}\label{coarse_fundamental} [Coarse Fundamental Signature  Class]
Let $(N, f, E, b )$  be  a  referenced  manifold  with  respect  to $M$.  The  coarse  fundamental  class   of  $b$ is  the  class   in  $K_{n-1}(\mathfrak{A}(X, G, M)/ \mathfrak{C}(X, G, M))$ of  the boundedly  controlled  operator  $b$ associated  to  the Hilbert-Poincar\'e   complex $E$. 

\end{definition}

We  interpret  now  the  main  result  of \cite{wulffdiplomarbeit} in  an  equivariant  setting: 

\begin{theorem}
The  coarse  fundamental  class is a referenced bordism  invariant. 
\end{theorem}
\begin{proof}
The  situation  is  completely  analogous  to  \cite{wulffdiplomarbeit},  where  the invariance  is  seen  to  be  a consequence  of  the  naturality of  the   assembly  map. 
Consider  the   diagram  of $G$-equivariant    inclusions, which  are  assumed  to  give   coarse  maps. 
$$ \xymatrix{N_1 \ar[r] &   \partial W \ar[d]   & \ar [l] N_2 \\   &  W \ar[d]^-{F} &  \\ & M & }$$ 

The  long  exact  sequence  in  $K$-theory  of  $C^*$-algebras  gives: 

$$\xymatrix{K_{p+1}  (\mathfrak{A}(W\diagup \partial  W)/ \mathfrak{C}(W\diagup \partial W)) \ar[r]^-{\partial} &   K_{p}  (\mathfrak{A}(\partial  W)/ \mathfrak{C}( \partial W )     \ar[d]^-{A_{\partial W}}  \ar[r] &  K_{p}  (\mathfrak{A}(W )/ \mathfrak{C}(W ) \ar[d]^-{A_{ W}}  \\   & K_p( C^*(\partial  W)   )\ar[r]_{i_{*}} & K_p(C^*(W))  }      $$ 
 Where  the  upper   morphism  $\partial  $  is  the  connecting  homomorphism,  and  the    vertical  morphisms  are coarse  assembly  maps. 
 
 The  functoriality   of  the  index  morphism,  assembly  map gives 
 $$F_{*}(A_W(i_1([b_1])))=  F_{*}(A_W(i_2([b_2]))). $$

\end{proof}

\begin{corollary}\label{lemmabordismreferences}
The coarse  fundamental  class  gives  a  group  homomorphism 
$$C:\Omega_n^{an, eq}(M)\longrightarrow  K_{n-1}(\mathfrak{A}(X, G, M)/ \mathfrak{C}(X, G, M) $$
\end{corollary}

\begin{definition}[Analytical  signature]

The  analytical  signature   of a  referenced  manifold  $(M, f, E, b )$ is  given  as  the  composition  of  the  coarse fundamental  class   morphism $C$ together  with  the  coarse  assembly  map.
(Recall that  the coarse assembly map for $X$ is the homomorphism
$$\mu: K_i^G(M) \cong K_{i+1}(D^*_G(M)/C^*_G(M)) \rightarrow K_i (C^*_G(M)),$$
where the first  instance  of  $K$  denotes equivariant $k$-homology  for  spaces,  and  all  the  others  $C^*$-algebra $K$-theory,   the  first isomorphism is given by Paschke duality, as written in \cite[p.242]{roeComparing}, and the second is the boundary map in the long exact sequence of $K$-groups associated to the ideal $C^*_G(M)$ in $D^*_G(M)$.)

\end{definition}

In the following we shorten the notation $\mathfrak{A}(X, G, M)$, $\mathfrak{C}(X, G, M)$ 
by $\mathfrak{A}(M)$, $\mathfrak{C}(M)$ respectively. 

Recall    that  a  directed  bordism,  in the  sense  of \ref{c-bordism}, produces an  algebraic  Hilbert-Poincar\'e  complex  with  boundary,  as  in the  sense  of  \ref{boundaryduality}. Hence, the  algebraic  signature  constructed  in \ref{AsmishSignature} is  well  defined  after  passing  to  \emph{geometric} bordism. 

\begin{definition}\label{SIGMA}
The  algebraic signature $\Sigma $   is  the  group  homomorphism
 
$$ \Omega_n^{\rm an, eq} (M)\longrightarrow K_n(C_r^*(G))$$
described  in \ref{AsmishSignature}.
The  fact  that  the   algebraic  signature  descends  to  the  referenced  bordism  groups  follow  from  the  fact that  a  referenced  bordism   gives  an  algebraic Hilbert-Poincar\'e complex with  boundary,  \ref{algebraicbordisminvariance} proves  that the algebraic  signature   is  the  same, and  thus  the   map  is  well defined  on  referenced  bordism  classes.

\end{definition}

\section{Mapping surgery  to  analysis }
In this  section,   we  will  state the  main  theorem  of  this  paper: 

\begin{theorem}\label{maintheorem}
Let  $M$  be  a  proper $G$  manifold   with a  bounded  isotropy  action. Assume  that  the quotient  
$M/G$ is  compact. 
Then,  we have the following homomorphism

$$\xymatrix{\Omega_n^{\rm an, eq} (M)   \ar[d]_{\rm C}  \ar[rrrd]^{\rm \Sigma } &  & &  &\\   K^{n-1}(\mathfrak{A}(M)/ \mathfrak{C}(M))  \ar[r]_-{\omega_1}   &  KK_G^{n}(C_0(M), \mathbb{C}  ) \ar[r]_-{\omega_2}  &  KK_G^n ( C_0(\eub{G}), \mathbb{C}) \ar[r]_-{\mu } & K_{n}(C^{r}_{*}(G))   & } ,$$
where   the  maps are  definded  as  follows: the  map $C$ is  the  coarse  fundamental  class, \ref{coarse_fundamental},  the  map   $\omega_1$  is  the isomorphism  constructed  in \cite{roeComparing}(denoted  by $\omega_4$  in  page  242),  the group   homomorphism $\omega_2$ is  induced  by  the  up  to  $G$-equivariant homotopy  unique  map  $M\to \eub{G}$,  and  $\mu $  denotes  the   analytical  Baum-Connes assembly  map   in  $KK$-theory. 
We  will  call  the  composition 
$$  \mu \circ w_2\circ w_1\circ C: \Omega_n^{\rm an, eq} (M)  \longrightarrow  K_{n}(C^{r}_{*}(G))   $$  
the  analytical   signature. 

\end{theorem}
\begin{proof}

The  analytical  Assembly  map $\mu: KK_G^n ( C_0(\eub{G}), \mathbb{C}) \to  K_n( C^r_*(G))$ is  given  by  the  composite  of  the 
\emph{descent  homomorphism}   
$$ KK_G^n ( C_0(\eub{G}),\mathbb{C})\to  KK^n(  C_0(\eub{G}) \rtimes_r G,  \mathbb{C}\rtimes_{r} G)$$ followed  by   composing  with  the  map  given  by  the   Kasparov  product  with the  Mishchenko-Fomenko  line  bundle for  $\eub{G}$,  $$KK^n(  C_0(\eub{G}) \rtimes_r G,  \mathbb{C}\rtimes_r G) \to KK^n(\mathbb{C}, C_r^*(G)).$$ 
By  $KK$-theoretical homotopy  invariance,  the  composite  map 
 $$KK_G^{n}(C_0(M), \mathbb{C}  ) \underset{\omega_2}{\longrightarrow}   KK_G^n ( C_0(\eub{G}), \mathbb{C}) \underset{\mu}{\longrightarrow } K_{n}(C^{r}_{*}(G))   $$ 
 agrees  with  the  composite    
 $$KK_G^{n}(C_0(M), \mathbb{C}  )\to KK^{n}(C_0(M)\rtimes_r G , \mathbb{C} \rtimes_r G  ) \to KK^n( \mathbb{C}, C_r^{*}(G)), $$
 which  consists  of  the   descent  homomorphism    followed  by  the Kasparov  product  with  a  Mishchenko-Fomenko element  for   $C_0(M)$ (called  $w_5$ and  $w_6$  in  \cite{roeComparing}, p. 242,  respectively.) 
 
By  \ref{algebraicbordisminvariance}, \ref{HomBordInvariance},  the bordism  relations  are  compatible.  

Finally, by  commutativity  of  the  diagram  in page  242  of  \cite{roeComparing},  the  assembly  maps   commute.  
\end{proof}

\typeout{----------------------------  linluesau.tex  ----------------------------}


\end{document}